\pdfoutput=1 

\RequirePackage[l2tabu, orthodox]{nag}

\documentclass[12pt,reqno]{amsart}
\usepackage[utf8]{inputenc} 
\usepackage{fullpage,url,amssymb,enumerate,colonequals,graphicx}
\usepackage[all]{xy} 
\usepackage{mathrsfs} 

    \DeclareFontFamily{U}{wncy}{}
    \DeclareFontShape{U}{wncy}{m}{n}{<->wncyr10}{}
    \DeclareSymbolFont{mcy}{U}{wncy}{m}{n}
    \DeclareMathSymbol{\Sha}{\mathord}{mcy}{"58} 

\usepackage[T1]{fontenc}
\usepackage[utf8]{inputenc}

\usepackage[dvipsnames,xcdraw]{xcolor}

\newcommand{\defi}[1]{\textsf{#1}} 

\newcommand{\Aff}{\mathbb{A}}

\newcommand{\F}{\mathbb{F}}

\newcommand{\PP}{\mathbb{P}}
\newcommand{\Q}{\mathbb{Q}}

\newcommand{\Z}{\mathbb{Z}}
\newcommand{\Qbar}{{\overline{\Q}}}

\DeclareMathOperator{\Char}{char}

\DeclareMathOperator{\Frac}{Frac}

\DeclareMathOperator{\Spec}{Spec}

\newcommand{\intersect}{\cap} 
\newcommand{\Intersection}{\bigcap} 
\newcommand{\isom}{\simeq}

\newcommand{\To}{\longrightarrow}
\newcommand{\union}{\cup} 
\newcommand{\Union}{\bigcup} 

\newtheorem{theorem}{Theorem}[section]

\newtheorem{corollary}[theorem]{Corollary}
\newtheorem{proposition}[theorem]{Proposition}

\theoremstyle{definition}

\newtheorem{question}[theorem]{Question}

\theoremstyle{remark}
\newtheorem{remark}[theorem]{Remark}

\makeatletter
\g@addto@macro\bfseries{\boldmath} 
\makeatother

\usepackage{microtype}  

\usepackage[
	pagebackref,
	pdfauthor={Bhargav Bhatt and Bjorn Poonen}, 
	pdftitle={Diophantine sets},
]{hyperref}

\begin{document}

\title{Diophantine sets}

\author{Bhargav Bhatt}
\address{School of Mathematics, Institute for Advanced Study \& Department of Mathematics, Princeton University}
\email{bhargav.bhatt@gmail.com}
\urladdr{\url{https://www.math.ias.edu/~bhatt/}}

\author{Bjorn Poonen}
\thanks{B.B.\ was partially supported by a fellowship from the Packard Foundation and grants from the Simons Foundations (MPS-SIM-00622511, MPS-PERF-00001529-02).  B.P.\ was supported in part by Simons Foundation grant \#402472.  The research was done while B.P.\ was visiting the 2025 program ``Definability, decidability, and computability'' at the Hausdorff Research Institute in Bonn.}
\address{Department of Mathematics, Massachusetts Institute of Technology, Cambridge, MA 02139-4307, USA}
\email{poonen@math.mit.edu}
\urladdr{\url{http://math.mit.edu/~poonen/}}

\date{November 20, 2025}

\begin{abstract}
Diophantine subsets of $\Z$ play a key role in the negative answer to Hilbert's tenth problem.
The definition of diophantine set generalizes in several ways to other commutative rings.
We compare these definitions.
Along the way, we prove that for every finitely presented scheme $Y$ over a ring $R$, there exists an \emph{affine} $R$-scheme $X$ with a finitely presented $R$-morphism $X \to Y$ such that $X(R') \to Y(R')$ is surjective for every $R$-algebra $R'$.
\end{abstract}

\maketitle

\section{Introduction}\label{S:introduction}

A subset $A \subset \Z$ is \defi{diophantine} if there exists $f \in \Z[t,x_1,\ldots,x_m]$ such that 
\[
	A = \{a \in \Z : \exists x \in \Z^m \textup{ such that } f(a,x)=0\}.
\]
In other words, if one thinks of $t$ as a parameter, then $f$ defines a $1$-parameter family of diophantine equations, and $A$ is the set of parameter values for which the specialized diophantine equation is solvable in the remaining variables.
Diophantine sets play an essential role in the negative answer to Hilbert's tenth problem: the key step is proving that every computably enumerable subset of $\Z$ is diophantine; see \cite{Matiyasevich1993} for an exposition.

There are several ways to generalize the definition to ground rings other than $\Z$, and they are not always equivalent.
The goal of this article is to sort out the relationships between various definitions; see Theorem~\ref{T:main}.

Let $R$ be a ring.  (All rings in this paper are assumed commutative.)
\begin{itemize}
\item Call $A \subset R^n$ \defi{existential} if $A = \{a \in R^n : \phi(a)\}$ for some existential first-order formula $\phi(t)$, that is, a formula $(\exists x_1) \cdots (\exists x_m) \; \psi(t,x)$, where $\psi(t,x)$ is a Boolean combination of polynomial equations in $t_1,\dots,t_n,x_1,\dots,x_m$.
\item \defi{Positive-existential} is the same except that the only logical operators allowed in the Boolean combination are $\land$ (and) and $\lor$ (or).
\item \defi{Many-polynomial-diophantine} is the same except that only $\land$ is allowed; in other words, there exist $r \in \Z_{\ge 0}$ and $f_1,\ldots,f_r \in R[t_1,\ldots,t_n,x_1,\ldots,x_m]$ such that 
\[
	A = \{a \in R^n : \exists x \in R^m \textup{ such that } f_1(a,x)=\cdots=f_r(a,x)=0\}.
\]
\item \defi{One-polynomial-diophantine} is the same except that there is just one equation, with no Boolean operators; in other words, there exists $f \in R[t_1,\ldots,t_n,x_1,\ldots,x_m]$ such that 
\[
	A = \{a \in R^n : \exists x \in R^m \textup{ such that } f(a,x)=0\}.
\]
See \cite[Definition~1.2.1]{Shlapentokh2007}.
\item Let $X$ be a finitely presented $R$-scheme.
Call $A \subset X(R)$ \defi{morphism-diophantine} if there exists an $R$-scheme $Y$ and finitely presented $R$-morphism $\phi \colon Y \to X$ such that $A = \phi(Y(R))$.
\end{itemize}

\begin{remark}
\label{R:zeros}
If $f \in R[t_1,\ldots,t_n]$, then the set of zeros of $f$ in $R^n$ is one-polynomial-diophantine (take $m=0$ in the definition).
Likewise, the common zero set of $f_1,\ldots,f_r \in R[t_1,\ldots,t_n]$ in $R^n$ is many-polynomial-diophantine.
\end{remark}

\begin{remark}
Any morphism of finitely presented $R$-schemes $X \to X'$ maps morphism-diophantine subsets of $X(R)$ to morphism-diophantine subsets of $X'(R)$ (use the same $Y$).
\end{remark}

\begin{remark}
The notion of morphism-diophantine can be applied to $A \subset R^n$ by taking $X = \Aff^n_R$.
\end{remark}

\begin{theorem}
\label{T:main}
Let $R$ be a ring. \hfill
\begin{enumerate}[\upshape (a)]
\item \label{I:one implies many}
For $A \subset R^n$, one-polynomial-diophantine implies many-polynomial-diophantine.
\item \label{I:many implies one}
The converse holds for all $n$ and all $A \subset R^n$ if and only if 
\begin{equation}
\label{E:only (0,0)}
   \textup{there exists $g \in R[x,y]$ whose zero set in $R^2$ is $\{(0,0)\}$.}
\end{equation}
\item \label{I:many implies positive}
For $A \subset R^n$, many-polynomial-diophantine implies positive-existential.
\item \label{I:positive implies many}
The converse holds if and only if $(R \times \{0\}) \union (\{0\} \times R)$ is many-polynomial-diophantine.
\item \label{I:positive implies existential}
For $A \subset R^n$, positive-existential implies existential.
\item \label{I:existential implies positive}
The converse holds if and only if $R-\{0\}$ is positive-existential.
\item \label{I:many equals morphism}
For $A \subset R^n$, many-polynomial-diophantine is equivalent to morphism-diophantine.
\end{enumerate}
\end{theorem}

Given \eqref{I:many equals morphism}, we propose that, going forward, \defi{diophantine} should be defined as morphism-diophantine (or equivalently, for $A \subset R^n$, many-polynomial-diophantine).

Most of Theorem~\ref{T:main} is easy and well known.
The only complicated part is \eqref{I:many equals morphism}, which will be deduced from Theorem~\ref{T:covered by affine}.

Finally, in Sections \ref{S:2}--\ref{S:two axes}, we explore for which rings the conditions in \eqref{I:many implies one}, \eqref{I:positive implies many}, and \eqref{I:existential implies positive} hold.
The most substantial new result in these sections is Theorem~\ref{T:Frac R algebraically closed}, which states that if $K$ is a field that is not an algebraic closure of a finite field, then there exists $R$ with $\Frac R = K$ such that \eqref{E:only (0,0)} holds.

Here is one application of all of the above, which we record for convenience of use:

\begin{corollary}
\label{C:all equivalent}
Suppose that $R$ is a domain satisfying any of the following:
\begin{enumerate}[\upshape (i)]
\item $R$ is a finitely generated $\Z$-algebra.
\item $R$ is a finitely generated algebra over a field $k$, but $R$ is not an algebraically closed field.
\item $R$ is a localization of one of the above.
\end{enumerate}
Then all five notions (existential, positive existential, many-polynomial-diophantine, one-polynomial-diophantine, and morphism-diophantine) are equivalent.
\end{corollary}

\begin{proof}
We check that the conditions for all the converses in Theorem~\ref{T:main} are satisfied: 

(b) If the domain $R$ is a finitely generated $\Z$-algebra, or a localization thereof, then $\Frac R$ is not algebraically closed.
If the domain $R$ is a finitely generated $k$-algebra with $\dim R > 0$, or a localization thereof, then $\Frac R$ is not algebraically closed.
If the domain $R$ is a finitely generated $k$-algebra and $\dim R = 0$, then $R$ is a field, and by assumption it is not algebraically closed.
Thus in all three cases, $\Frac R$ is not algebraically closed.
By Proposition~\ref{P:Frac R not algebraically closed}, \eqref{E:only (0,0)} holds.

(d) Since $R$ is a domain, $(R \times \{0\}) \union (\{0\} \times R)$ is the zero set of $xy$, hence one-polynomial-diophantine, hence many-polynomial-diophantine.

(f) If $\dim R = 0$, then $R$ is a field, so $R - \{0\}$ is defined by $(\exists x) \; tx=1$, hence positive-existential.
If $\dim R > 0$, then the function field $\Frac R$ cannot carry a nontrivial henselian valuation, so $R - \{0\}$ is positive-existential by work of Moret-Bailly; see Theorem~\ref{T:Moret-Bailly definition of nonvanishing}\eqref{I:not local henselian}.
\end{proof}

\section{Covering by an affine scheme}

\begin{theorem}
\label{T:covered by affine}
Let $R$ be a ring.
Let $Y$ be a finitely presented $R$-scheme.
Then there exists an \emph{affine} $R$-scheme $X$ with a finitely presented $R$-morphism $X \to Y$ such that $X(R) \to Y(R)$ is surjective.
\end{theorem}

\begin{remark}
Let $Y_1,\ldots,Y_n$ be an affine open cover of $Y$.
If $R$ is local, then the disjoint union $X \colonequals \sqcup Y_i$ works.
But if $R$ is not local, $Y(R)$ need not equal $\Union Y_i(R)$: a morphism $\Spec R \to Y$ does not necessarily have image contained in a single $Y_i$.
\end{remark}

\begin{remark}
\label{R:Jouanolou}
If $Y$ is quasi-projective over $R$ (or more generally $Y$ has an ample family of line bundles), then Jouanolou's trick (as generalized by Thomason \cite[Proposition~4.4]{Weibel1989}) produces a vector bundle $E \to Y$ and an $E$-torsor $X \to Y$ with $X$ affine, and then $X(R) \to Y(R)$ is surjective since vector bundles over affine schemes have no nontrivial torsors.
\end{remark}

\begin{proof}[Proof of Theorem~\ref{T:covered by affine}]
Let $I=Y(R)$.
Let $X' = \sqcup_I \Spec R$ and $X'' = \Spec\left(\prod_I R \right)$.
There is an obvious $R$-morphism $X' \to Y$ with $X'(R) \to Y(R)$ surjective by construction.
Since $Y$ is finitely presented over a ring, $Y$ is quasi-compact and quasi-separated, so \cite[Theorem~1.3]{Bhatt2016} implies that $X' \to Y$ factors uniquely as $X' \to X'' \to Y$.
By writing $\prod_I R$ as a filtered colimit of finitely presented $R$-algebras and using the finite presentation property of $Y$, we find a further factorization $X' \to X'' \to X \to Y$, where $X$ is an affine, finitely presented $R$-scheme.
The composite morphism was surjective on $R$-points, so $X \to Y$ is too.
\end{proof}

Although we do not need it for our application, we can also prove the following strengthening:

\begin{theorem}
\label{T:covered by affine 2}
Let $R$ be a ring.
Let $Y$ be a finitely presented $R$-scheme.
Then there exists an affine $R$-scheme $X$ with a finitely presented $R$-morphism $X \to Y$ such that $X(R') \to Y(R')$ is surjective \emph{for every $R$-algebra $R'$}.
\end{theorem}

\begin{proof}
Since $Y$ is finitely presented over $R$, for each $R'$, each element of $Y(R')$ comes from an element of $Y(R_0)$ for some finitely presented $R$-algebra $R_0$.
Thus it suffices to find $X$ such that $X(R_0) \to Y(R_0)$ is surjective for every finitely presented $R$-algebra $R_0$.
Let $\mathscr{R}$ be a set of representatives of the isomorphism classes of finitely presented $R$-algebras.
Let $X' = \sqcup_{R_0 \in \mathscr{R}} \sqcup_{y \in Y(R_0)} \Spec R_0$ and $X'' = \Spec\left(\prod_{R_0 \in \mathscr{R}} \prod_{y \in Y(R_0)} R_0 \right)$.
Produce a factorization $X' \to X'' \to X \to Y$ as before.
For every $R_0$, the morphism $X' \to Y$ is surjective on $R_0$-points, so $X \to Y$ is too.
\end{proof}

\begin{remark}
Theorem~\ref{T:covered by affine 2} says that if $Y$ is a finitely presented scheme over a ring, then the category of affine schemes equipped with a morphism to $Y$ has a \emph{weak terminal object} $X$ (every object maps to $X$, but not necessarily uniquely), and $X$ can be taken to be finitely presented over $Y$.
\end{remark}

\begin{remark}
\label{R:algebraic space}
The proofs of Theorems \ref{T:covered by affine} and~\ref{T:covered by affine 2} generalize to the situation where $Y$ is an \emph{algebraic space} that is finitely presented over $R$.
\end{remark}

\section{Comparing the definitions}

\begin{proof}[Proof of Theorem~\ref{T:main}]
\hfill
\begin{enumerate}[\upshape (a)]
\item Trivial.

\item
$\Rightarrow$: Suppose that every many-polynomial-diophantine set is one-polynomial-diophantine.  Then the many-polynomial-diophantine subset $\{(0,0)\} \subset R^2$ defined by $t_1=t_2=0$ is one-polynomial-diophantine.
That is, there exists $f(t_1,t_2,x_1,\ldots,x_m)$ whose zero set contains $(0,0,u)$ for some $u \in R^m$ but contains no points with $(t_1,t_2) \ne (0,0)$.
Then $g(t_1,t_2) \colonequals f(t_1,t_2,u)$ has zero set $\{(0,0)\}$.

\noindent $\Leftarrow$:
Suppose that $g(x,y)$ has zero set $\{(0,0)\}$.
We prove by induction on $r$ that the common zero set of any $r$ polynomials $f_1,\ldots,f_r$ is the zero set of a single polynomial $f$.
If $r=0$, take $f=0$.
If $r=1$, take $f=f_1$.
If $r \ge 2$, apply the inductive hypothesis to $g(f_1,f_2),f_3,\ldots,f_r$.

\item Trivial.

\item 
$\Rightarrow$: 
The $(R \times \{0\}) \union (\{0\} \union R) = \{(x,y) \in R^2: x=0 \lor y=0\}$ is positive-existential.
If positive-existential implies many-polynomial-diophantine, 
then it is many-polynomial-diophantine too.

\noindent $\Leftarrow$: In $\phi(t)$, convert each expression $f(t,x) = 0 \lor g(t,x)=0$ to $f(t,x) = z \land g(t,x) = w \land (z=0 \lor w=0)$ and replace $z=0 \lor w=0$ with its many-polynomial-diophantine definition.
Eventually, we rewrite $\phi(t)$ in the form $(\exists x_1) \cdots (\exists x_m) \psi(t,x)$ for some new $m$ and $\psi$ such that $\psi$ is a finite conjunction of polynomial equations.

\item Trivial.

\item 
$\Rightarrow$: 
The set $R - \{0\} = \{x \in R : x \ne 0 \}$ is existential.
If existential implies positive-existential, then it is positive-existential too.

\noindent $\Leftarrow$: In $\phi(t)$, convert each expression $f(t,x) \ne 0$ to $f(t,x) = z \land z \ne 0$ and replace $z \ne 0$ with the equivalent positive-existential formula.

\item 
$\Rightarrow$: 
To say that $A \subset R^n$ is many-polynomial-diophantine
is to say that $A$ is the image of the $R$-points
under the projection morphism 
\[
	\Spec \frac{R[t_1,\ldots,t_n,x_1,\ldots,x_m]}{(f_1,\ldots,f_r)} \To \Spec R[t_1,\ldots,t_n] = \Aff^n_R
\]
for some $m$ and $r$ and $f_1,\ldots,f_r \in R[t_1,\ldots,t_n,x_1,\ldots,x_m]$.

\noindent $\Leftarrow$: Suppose that $A \subset R^n$ is morphism-diophantine.
By definition, $A = \phi(Y(R))$ for some morphism $\phi \colon Y \to \Aff^n_R$
of finitely presented $R$-schemes.
Theorem~\ref{T:covered by affine} provides a morphism $\psi \colon X \to Y$ of finitely presented $R$-schemes such that $\psi(X(R)) = Y(R)$.
Then $A = (\phi \psi)(X(R))$ and $X$ is affine and finitely presented, 
so $A$ is many-polynomial-diophantine.
\qedhere
\end{enumerate}
\end{proof}

\begin{remark}
In the definition of morphism-diophantine for $A \subset X(R)$, we required the $Y$ mapping to $X$ to be a scheme.
Remark~\ref{R:algebraic space} implies that allowing $Y$ to be an algebraic space would yield an equivalent definition.
One could also define morphism-diophantine for $A \subset X(R)$ with $X$ a finitely presented $R$-algebraic space.
\end{remark}

\section{Two-variable polynomials vanishing only at \texorpdfstring{$(0,0)$}{(0,0)}}
\label{S:2}

Recall that \eqref{E:only (0,0)} is the property that there exists $g \in R[x,y]$ whose zero set is $\{(0,0)\}$.

\begin{proposition}[\protect{cf.~\cite[Lemma~1.2.3]{Shlapentokh2007}}]
\label{P:Frac R not algebraically closed}
If $R$ is a domain whose fraction field $\Frac R$ is not algebraically closed, then \eqref{E:only (0,0)} holds.
\end{proposition}

\begin{proof}
Let $K=\Frac R$.
Since $K$ is not algebraically closed, there is a nontrivial finite extension $L/K$.
Let $\alpha \in L-K$.
Let $h(x,y)$ be the norm form $\operatorname{N}_{L/K}(x+y\alpha) \in K[x,y]$.
The only zero of $h$ in $K^2$ is $(0,0)$.
Let $g = r h$ where $r \in R - \{0\}$ is chosen so that $g \in R[x,y]$.
Then the only zero of $g$ in $R^2$ is $(0,0)$.
\end{proof}

For domains whose fraction field $K$ is algebraically closed, the answer is sometimes yes, sometimes no, as the next two propositions show.

\begin{proposition} 
\label{P:K algebraically closed}
If $K$ is an algebraically closed field, then \eqref{E:only (0,0)} fails.
\end{proposition}

\begin{proof}
If the zero set of $g \in K[x,y]$ is nonempty, its Krull dimension is at least $2-1>0$.
\end{proof}

If $K$ is algebraic over a finite field, then any domain $R$ with fraction field $K$ equals $K$, and Proposition~\ref{P:K algebraically closed} applies.
Otherwise, we have the following:

\begin{theorem}
\label{T:Frac R algebraically closed}
If $K$ is a field that is not an algebraic closure of a finite field, there exists a domain $R$ with fraction field $K$ for which \eqref{E:only (0,0)} holds.
\end{theorem}

\begin{proof}
Because of Proposition~\ref{P:Frac R not algebraically closed},
we may assume that $K$ is algebraically closed.

First suppose that $K$ is countable.
If $\Char K=0$, let $A_0=\Z$ and $F_0=\Q$.
If $\Char K = p>0$, let $A_0=\F_p[t]$ and $F_0=\F_p(t)$.
Since $K$ is not an algebraic closure of a finite field, we may view $K$ as an extension of $F_0$.
Let $C$ be a smooth plane quartic in $\PP^2_{F_0}$ passing through $(1:0:0)$;
if $\Char K>0$, assume moreover that $C$ is not isotrivial.
The genus of $C$ is at least $2$ (it is $3$), so the Mordell conjecture \cite{Faltings1983} and its function field analogues \cite{Grauert1965,Samuel1966} imply that $C(F)$ is finite for any finitely generated extension $F \supset F_0$.

Let $\Aff^2_{F_0} \subset \PP^2_{F_0}$ be the standard affine patch containing $(1:0:0)$.
Let $C' = C \intersect \Aff^2_{F_0}$, which has an equation $g(x,y)=0$.
Then $C'(F_0)$ consists of $(0,0)$ and finitely many other points in $F_0^2$. 
By scaling coordinates, we may assume none of these other points lie in $A_0^2$.
Let $C'$ denote also the hypersurface in $\Aff^2_{A_0}$ defined by $g(x,y)=0$.
Thus $C'(A_0)=\{(0,0)\}$.

Write $K=\{\alpha_1,\alpha_2,\ldots\}$.
Let $F_n=F_0(\alpha_1,\ldots,\alpha_n)$.
By induction, we will construct finitely generated subrings $A_0 \subset A_1 \subset \cdots$ such that $\Frac A_n=F_n$ and $C'(A_n)=\{(0,0)\}$.
Suppose that $A_n$ has been constructed.
Write $A=A_n$, $F=F_n$, $\alpha = \alpha_{n+1}$, so $F_{n+1}=F(\alpha)$.

\emph{Case 1: $\alpha$ is algebraic over $F$.}
Since $F_{n+1}$ is finitely generated, $C(F_{n+1})$ is finite.
By scaling $\alpha$, we may assume that $\alpha$ is integral over $A$.
We will let $A_{n+1} = A + c A[\alpha]$ for a suitable $c \in A - \{0\}$.
We have $\Intersection_{c \in A-\{0\}} (A + c A[\alpha]) = A \subset F$, so for each of the finitely many $P \in C'(F_{n+1})-C'(F)$, we can find $c_P$ such that $P \notin (A + c_P A[\alpha])^2$.
Let $c=\prod_P c_P$.
Let $A_{n+1} = A + c A[\alpha]$.
Then $P \notin A_{n+1}^2$ for all $P$.
Hence $C'(A_{n+1})-C'(F) = \emptyset$,
so 
\[
	C'(A_{n+1})=C'(A_{n+1}) \intersect C'(F) = C'(A_{n+1} \intersect F) = C'(A) = \{(0,0)\}.
\]

\emph{Case 2: $\alpha$ is transcendental over $F$.}
Every rational map $\PP^1 \to C'$ is constant, so $C'(F(\alpha))=C'(F)$.
Thus if $A_{n+1} \colonequals A[\alpha]$, then $C'(A_{n+1})=C'(A)=\{(0,0)\}$.

\medskip

Now that we have all the $A_n$, let $R = \Union_{n \ge 0} A_n$.
Then $\Frac R = \Union_{n \ge 0} F_n = K$, 
and $C'(R) = \Union_{n \ge 0} C'(A_n) = \{(0,0)\}$.
This completes the proof when $K$ is countable.

\medskip

It remains to prove the result for one algebraically closed field of each characteristic and each uncountable cardinality $\kappa$.
By taking an ultrapower of a countable field (for which we already know the result), we find $K$ with $\#K \ge \kappa$ for which $R$ and $g$ exist.
By the theory of transcendence bases, we can find an algebraically closed subfield $K_0 \subset K$ of the desired cardinality $\kappa$, and we may assume that $K_0$ contains the coefficients of $g$.
Unfortunately, $R_0 \colonequals R \intersect K_0$ might have fraction field smaller than $K_0$.
To deal with this, we will enlarge $R_0$ and $K_0$.
Represent each element of $K$ as a quotient of elements of $R$,
so that we can speak of its numerator and denominator.
Given any subfield $L \subset K$, let $L'$ be the extension of $L$ generated by the numerators and denominators of all the elements of $L$, and let $\phi(L)$ be the algebraic closure of $L'$ in $K$.
For $n \in \Z_{\ge 0}$, let $K_{n+1} = \phi(K_n)$.
By induction, $\# K_n = \kappa$ for all $n$.
Let $K_\omega = \Union K_n$, so $\#K_\omega = \kappa$.
Let $R_\omega = R \intersect K_\omega$.
Each $a \in K_\omega$ lies in $K_n$ for some $n$,
so its numerator and denominator lie in $K_{n+1} \subset K_\omega$,
hence in $R_\omega$.
Thus $\Frac R_\omega = K_\omega$.
Also, the only solution to $g(x,y)=0$ in $R^2$ is $(0,0)$,
so the same is true in $R_\omega^2$.
In other words, the result holds for $K_\omega$.
\end{proof}

\begin{proposition}
Let $R$ be a normal domain
with $\Frac R$ isomorphic to $\Qbar$ or $\overline{\F_p(t)}$.
Then \eqref{E:only (0,0)} fails for $R$.
\end{proposition}

\begin{proof}
This is a consequence of work of Rumely \cite{Rumely1986}, extended by Moret-Bailly \cite{Moret-Bailly1989}.
Let $g \in R[x,y]$ be such that $g(0,0)=0$.
We will show that $g(x,y)=0$ has infinitely many solutions over $R$.

Factor $g$ over the algebraically closed field $\Frac R$; 
some factor vanishes at $(0,0)$ and may be scaled to assume that
it has coefficients in $R$.
Without loss of generality, replace $g$ by this factor.
Now $g$ is geometrically irreducible.

Let $R_0$ be the ring generated by the coefficients of $g$, together with one transcendental element of $R$ if $\Frac R = \overline{\F_p(t)}$.
Let $R_1$ be the integral closure of $R_0$ in its fraction field,
so $R_0 \subset R_1 \subset R$.
Then $R_1$ is an excellent Dedekind domain satisfying condition~(\textbf{T}) of \cite{Moret-Bailly1989}.
The morphism $\Spec R_1[x,y]/(g) \to \Spec R_1$ is surjective because of the $(0,0)$ section.
Applying \cite[\S6]{Moret-Bailly1989} to $R_1$ proves that $g(x,y)=0$ has infinitely many solutions over $R_1$, and hence also over $R$.
\end{proof}

\begin{question}
Does \eqref{E:only (0,0)} hold for \emph{some} normal domain with algebraically closed fraction field?
\end{question}

\section{Defining the nonzero elements}
\label{S:nonzero elements}

\begin{theorem}[Moret-Bailly]
\label{T:Moret-Bailly definition of nonvanishing}
Let $R$ be a noetherian ring.
\begin{enumerate}[\upshape (a)]
\item \label{I:not local henselian}
If $R$ is a domain that is not local henselian, then $R-\{0\}$ is positive-existential.
\item
If $R$ is a localization of a noetherian Jacobson ring, then $R-\{0\}$ is positive-existential.
\item 
If $R$ is a henselian excellent local ring and $\dim R > 0$, then $R-\{0\}$ is \emph{not} positive-existential.
\end{enumerate}
\end{theorem}

\begin{proof}
For the ring of integers in a number field, see \cite[Proposition~1(b)]{Denef-Lipshitz1978}.
For the ring of $S$-integers in a global field (even when $S$ is infinite), see \cite[Proposition~2.2.4]{Shlapentokh2007}.
The general case is \cite{Moret-Bailly2007}.
\end{proof}

See \cite{Moret-Bailly2007} for further results along these lines.

\section{Defining the union of two axes}
\label{S:two axes}

If $R$ is a domain, then $(R \times \{0\}) \union (\{0\} \times R)$ is one-polynomial-diophantine, since it is the set of solutions to $xy=0$.
For non-domains, sometimes $(R \times \{0\}) \union (\{0\} \times R)$ is many-polynomial-diophantine (necessarily via polynomials different from $xy$), and sometimes it is not: 

\begin{proposition}
\label{P:two axes}
For a ring $R$, the following are equivalent:
\begin{enumerate}[\upshape (i)]
\item \label{I:not a product}
       $R$ is not a product of two nonzero rings.
\item \label{I:connected or empty}
       $\Spec R$ is connected or empty.
\item \label{I:finite union}
      Any finite union of many-polynomial-diophantine sets is many-polynomial-diophantine.
\item \label{I:two axes}
The set $(R \times \{0\}) \union (\{0\} \times R)$ is many-polynomial-diophantine.
\end{enumerate}
\end{proposition}

\begin{proof} \hfill

\eqref{I:not a product}$\Rightarrow$\eqref{I:connected or empty}:
(This is well-known.)
We prove the contrapositive.
Suppose that $\Spec R$ is the disjoint union of nonempty open sets $U_0$ and $U_1$.
For $i=0,1$, let $r_i \in R$ be the element that is $1$ on $U_i$ and $0$ on $U_{1-i}$.
Thus $r_0$ and $r_1$ are idempotents summing to $1$, 
so $R = r_0 R \times r_1 R$.

\eqref{I:connected or empty}$\Rightarrow$\eqref{I:finite union}:
If $\Spec R$ is empty, then $R=0$ and \eqref{I:finite union} holds.

Now assume that $\Spec R$ is connected.
It suffices to consider the union of \emph{two} many-polynomial-diophantine sets of $R^n$.
Write the first as the image of $X_0(R)$ under the projection $\Aff^{n+m} \to \Aff^n$ for some finitely presented closed subscheme $X_0 \subset \Aff^{n+m}$.
Similarly, write the second as the image of $X_1(R)$ for some $X_1 \subset \Aff^{n+m}$; we may assume that the $m$ is the same, by enlarging the smaller $m$ if necessary.
Let $Y$ be the (disjoint) union of $Y_0 \colonequals X_0 \times \{0\}$ and $Y_1 \colonequals X_1 \times \{1\}$ in $\Aff^{n+m+1}$.
Since $\Spec R$ is connected, any morphism $\Spec R \to Y$ has image contained in either $Y_0$ or $Y_1$.
That is, $Y(R) = Y_0(R) \union Y_1(R)$.
Thus the image of $Y(R)$ under the projection $\Aff^{n+m+1} \to \Aff^n$ 
is the union of the two given sets.
Finally, $Y$ is a finitely presented closed subscheme of $\Aff^{n+m+1}$,
so the image of $Y(R)$ is many-polynomial-diophantine.

\eqref{I:finite union}$\Rightarrow$\eqref{I:two axes}:
Both $R \times \{0\}$ and $\{0\} \times R$ are many-polynomial-diophantine (even one-polynomial diophantine).
By \eqref{I:finite union}, their union is many-polynomial-diophantine.

\eqref{I:two axes}$\Rightarrow$\eqref{I:not a product}:
We prove the contrapositive.
Suppose that $R \isom R_1 \times R_2$ with $R_1,R_2 \ne 0$.
Call $A \subset R^n = R_1^n \times R_2^n$ a product set if it is $A_1 \times A_2$
for some sets $A_i \subset R_i^n$.
Each $f \in R[x_1,\ldots,x_n]$ amounts to a pair of functions $f_i \in R_i[x_1,\ldots,x_n]$,
so the zero set of $f$ is a product set.
Intersecting product sets yields a product set,
and the image of a product set under a projection $R^{n+m} \to R^n$
is a product set.
Thus every many-polynomial-diophantine subset of $R^n$ is a product set.
But $(R \times \{0\}) \union (\{0\} \times R)$ is not a product set,
because, viewed in $R_1^2 \times R_2^2$,
it contains $((1,0),(1,0))$ and $((0,1),(0,1))$
but not $((1,0),(0,1))$.
Thus $(R \times \{0\}) \union (\{0\} \times R)$ is not many-polynomial-diophantine.
\end{proof}


\section*{Acknowledgments} 

We thank Laurent Moret-Bailly for several comments and corrections, 
and in particular for pointing out Remark~\ref{R:Jouanolou}.

\bibliographystyle{amsalpha}
\bibliography{diophantine} 

@article {Bhatt2016,
    AUTHOR = {Bhatt, Bhargav},
     TITLE = {Algebraization and {T}annaka duality},
   JOURNAL = {Camb. J. Math.},
  FJOURNAL = {Cambridge Journal of Mathematics},
    VOLUME = {4},
      YEAR = {2016},
    NUMBER = {4},
     PAGES = {403--461},
      ISSN = {2168-0930,2168-0949},
   MRCLASS = {14F05 (14A20 18A30)},
  MRNUMBER = {3572635},
MRREVIEWER = {Hsian-Hua\ Tseng},
       DOI = {10.4310/CJM.2016.v4.n4.a1},
       URL = {https://doi.org/10.4310/CJM.2016.v4.n4.a1},
}

@article {Denef-Lipshitz1978,
    AUTHOR = {Denef, J. and Lipshitz, L.},
     TITLE = {Diophantine sets over some rings of algebraic integers},
   JOURNAL = {J. London Math. Soc. (2)},
  FJOURNAL = {Journal of the London Mathematical Society. Second Series},
    VOLUME = {18},
      YEAR = {1978},
    NUMBER = {3},
     PAGES = {385--391},
      ISSN = {0024-6107,1469-7750},
   MRCLASS = {12L05 (03D35 10B99)},
  MRNUMBER = {518221},
MRREVIEWER = {P.\ Roquette},
       DOI = {10.1112/jlms/s2-18.3.385},
       URL = {https://doi.org/10.1112/jlms/s2-18.3.385},
}

@article {Faltings1983,
    AUTHOR = {Faltings, G.},
     TITLE = {Endlichkeitss\"atze f\"ur abelsche {V}ariet\"aten \"uber
              {Z}ahlk\"orpern},
   JOURNAL = {Invent. Math.},
  FJOURNAL = {Inventiones Mathematicae},
    VOLUME = {73},
      YEAR = {1983},
    NUMBER = {3},
     PAGES = {349--366},
      ISSN = {0020-9910,1432-1297},
   MRCLASS = {11D41 (11G30 14G25)},
  MRNUMBER = {718935},
MRREVIEWER = {James\ Milne},
       DOI = {10.1007/BF01388432},
       URL = {https://doi.org/10.1007/BF01388432},
      note = {English translation: Finiteness theorems for abelian varieties over number fields, 9--27 in \emph{Arithmetic Geometry (Storrs, Conn., 1984)}, Springer, New York, 1986.  Erratum in: Invent.\ Math.\ \textbf{75} (1984), 381},
}

@article {Grauert1965,
    AUTHOR = {Grauert, Hans},
     TITLE = {Mordells {V}ermutung \"uber rationale {P}unkte auf
              algebraischen {K}urven und {F}unktionenk\"orper},
   JOURNAL = {Inst. Hautes \'Etudes Sci. Publ. Math.},
  FJOURNAL = {Institut des Hautes \'Etudes Scientifiques. Publications
              Math\'ematiques},
    NUMBER = {25},
      YEAR = {1965},
     PAGES = {131--149},
      ISSN = {0073-8301,1618-1913},
   MRCLASS = {14.40 (10.00)},
  MRNUMBER = {222087},
MRREVIEWER = {M.\ Miwa},
       URL = {http://www.numdam.org/item?id=PMIHES_1965__25__131_0},
}

@book {Matiyasevich1993,
    AUTHOR = {Matiyasevich, Yuri V.},
     TITLE = {Hilbert's tenth problem},
    SERIES = {Foundations of Computing Series},
      NOTE = {Translated from the 1993 Russian original by the author,
              With a foreword by Martin Davis},
 PUBLISHER = {MIT Press, Cambridge, MA},
      YEAR = {1993},
     PAGES = {xxiv+264},
      ISBN = {0-262-13295-8},
   MRCLASS = {03-01 (03D35 11U05 68-01)},
  MRNUMBER = {1244324},
MRREVIEWER = {Cristian\ Calude},
}

@article {Moret-Bailly1989,
    AUTHOR = {Moret-Bailly, Laurent},
     TITLE = {Groupes de {P}icard et probl\`emes de {S}kolem. {I}},
   JOURNAL = {Ann. Sci. \'Ecole Norm. Sup. (4)},
  FJOURNAL = {Annales Scientifiques de l'\'Ecole Normale Sup\'erieure.
              Quatri\`eme S\'erie},
    VOLUME = {22},
      YEAR = {1989},
    NUMBER = {2},
     PAGES = {161--179},
      ISSN = {0012-9593},
   MRCLASS = {11G35 (11D72 14G25)},
  MRNUMBER = {1005158},
MRREVIEWER = {Philippe\ Satg\'e},
       URL = {http://www.numdam.org/item?id=ASENS_1989_4_22_2_161_0},
}

@article {Moret-Bailly2007,
    AUTHOR = {Moret-Bailly, Laurent},
     TITLE = {Sur la d\'efinissabilit\'e{} existentielle de la
              non-nullit\'e{} dans les anneaux},
   JOURNAL = {Algebra Number Theory},
  FJOURNAL = {Algebra \& Number Theory},
    VOLUME = {1},
      YEAR = {2007},
    NUMBER = {3},
     PAGES = {331--346},
      ISSN = {1937-0652,1944-7833},
   MRCLASS = {13B40 (11U09 13F40 13J15 14B12)},
  MRNUMBER = {2361937},
MRREVIEWER = {D.-M.\ Popescu},
       DOI = {10.2140/ant.2007.1.331},
       URL = {https://doi.org/10.2140/ant.2007.1.331},
}

@article {Rumely1986,
    AUTHOR = {Rumely, Robert S.},
     TITLE = {Arithmetic over the ring of all algebraic integers},
   JOURNAL = {J. Reine Angew. Math.},
  FJOURNAL = {Journal f\"ur die Reine und Angewandte Mathematik. [Crelle's
              Journal]},
    VOLUME = {368},
      YEAR = {1986},
     PAGES = {127--133},
      ISSN = {0075-4102,1435-5345},
   MRCLASS = {11D72 (11G35 11U09)},
  MRNUMBER = {850618},
MRREVIEWER = {Robert\ M.\ Guralnick},
       DOI = {10.1515/crll.1986.368.127},
       URL = {https://doi.org/10.1515/crll.1986.368.127},
}

@article {Samuel1966,
    AUTHOR = {Samuel, Pierre},
     TITLE = {Compl\'ements \`a{} un article de {H}ans {G}rauert sur la
              conjecture de {M}ordell},
   JOURNAL = {Inst. Hautes \'Etudes Sci. Publ. Math.},
  FJOURNAL = {Institut des Hautes \'Etudes Scientifiques. Publications
              Math\'ematiques},
    NUMBER = {29},
      YEAR = {1966},
     PAGES = {55--62},
      ISSN = {0073-8301,1618-1913},
   MRCLASS = {14.49 (14.40)},
  MRNUMBER = {204430},
MRREVIEWER = {John\ V.\ Armitage},
       URL = {http://www.numdam.org/item?id=PMIHES_1966__29__55_0},
}

@book {Shlapentokh2007,
    AUTHOR = {Shlapentokh, Alexandra},
     TITLE = {Hilbert's tenth problem},
    SERIES = {New Mathematical Monographs},
    VOLUME = {7},
      NOTE = {Diophantine classes and extensions to global fields},
 PUBLISHER = {Cambridge University Press, Cambridge},
      YEAR = {2007},
     PAGES = {xiv+320},
      ISBN = {978-0-521-83360-8; 0-521-83360-4},
   MRCLASS = {11U05 (03-02 03B25 11-02)},
  MRNUMBER = {2297245},
MRREVIEWER = {Jeroen\ Demeyer},
}

@incollection {Weibel1989,
    AUTHOR = {Weibel, Charles A.},
     TITLE = {Homotopy algebraic {$K$}-theory},
 BOOKTITLE = {Algebraic {$K$}-theory and algebraic number theory
              ({H}onolulu, {HI}, 1987)},
    SERIES = {Contemp. Math.},
    VOLUME = {83},
     PAGES = {461--488},
 PUBLISHER = {Amer. Math. Soc., Providence, RI},
      YEAR = {1989},
      ISBN = {0-8218-5090-3},
   MRCLASS = {18F25 (19D25)},
  MRNUMBER = {991991},
MRREVIEWER = {Barry\ H.\ Dayton},
       DOI = {10.1090/conm/083/991991},
       URL = {https://doi.org/10.1090/conm/083/991991},
}

\end{document}